\documentclass[a4paper,10pt,,reqno]{amsart}

\usepackage{amssymb} 
\usepackage{indentfirst}
\usepackage[colorlinks=true,urlcolor=green,citecolor=red,linkcolor=blue]{hyperref}

\numberwithin{equation}{section}

\newtheorem{theorem}{Theorem}[section]

\newtheorem{lemma}[theorem]{Lemma}
\newtheorem{proposition}[theorem]{Proposition}

\newtheorem{question}{Question}

\newtheorem{bigthm}{Theorem}

\newtheorem{innercustomthm}{Theorem}
\newenvironment{customthm}[1]
  {\renewcommand\theinnercustomthm{#1}\innercustomthm}
  {\endinnercustomthm}

\newcommand{\ignor}[1]{} 

\DeclareMathOperator{\Ker}{Ker}

\let\leq=\leqslant
\let\geq=\geqslant

\emergencystretch=\maxdimen
\begin{document}


\title[Profinite groups in which  centralizers are abelian]
{Profinite groups  in which  centralizers are abelian}

\author{Pavel Shumyatsky}
\address{Department of Mathematics, University of Brasilia, Brazil}
\email{pavel@unb.br, pz@mat.unb.br, \textrm{and} zapata@mat.unb.br}
\author{Pavel Zalesskii}
\author{Theo Zapata}




\maketitle


\begin{abstract}
The article deals with profinite groups in which the centralizers are abelian 
(CA-groups), that is, with profinite commutativity-transitive groups. 
It is shown that such groups are virtually pronilpotent. 
More precisely, let $G$ be a profinite CA-group. 
It is shown that $G$ has a normal open subgroup $N$ which is either abelian or pro-$p$. 
Further, a rather detailed information about the finite quotient $G/N$ is obtained.
\end{abstract}



\baselineskip 16.0pt

\section{Introduction}
\label{s:intro}

The simple concept of the transitivity of commutation among nontrivial elements have always attracted a lot of interest in Group Theory, and even in other areas of  Algebra. 
The condition is equivalent to the fact that all centralizers in a group are abelian. 
Thus groups with this property are called CA-groups, or CT-groups.

Finite (and locally finite) CA-groups were classified by several mathematicians from 1925 to 1998.
First, finite CA-groups were shown to be simple or solvable by Weisner \cite{Weisner:25}.
Then in the Brauer--Suzuki--Wall theorem \cite{BSW:58} finite CA-groups of even order were shown to be Frobenius groups, abelian groups, or two dimensional projective special linear groups over a finite field of even order, $\textup{PSL}(2,2^m)$ for $m\geq2$.
We highlight the work of M.~Suzuki \cite{Suzuki:57} which established that simple CA-groups have even order.
It should be mentioned that Suzuki's work 
was an important precursor of 
the outstanding Feit--Thompson theorem \cite{FT:63} on the solubility of groups of odd order.
A detailed description of (locally) finite Frobenius groups appearing as CA-groups was obtained by Yu-Fen~W. \cite{Wu:98}.

The classification of finite CA-groups is the following
(\textit{cf}. \protect{\cite[Ch.~6,~\S 2, Examples 1 and 2]{Suzuki:86}}):

\begin{theorem} \label{t:finite_CA}
A finite CA-group is either abelian, or a Frobenius group with abelian kernel and cyclic complement,
or isomorphic to $\textup{PSL}(2,2^m)$ for some integer~$m\geq 2$.
\end{theorem}

The current importance of
infinite CA-groups
is in its deep relation with residually free groups. Namely, finitely generated residually  free CA-groups are limit groups that played the key role in the solutions of Tarski problems. In fact, each centralizer in these groups is also malnormal 
(\textit{i.e.}, a conjugately separated subgroup in the sense that
its conjugate by any element outside the subgroup intersects it trivially), 
such groups are called \textit{CSA-groups}.
Amongst examples of infinite CA-groups 
we mention:
free groups,
torsion-free hyperbolic groups,
certain subgroups of $\textup{PSL}(2,\mathbb{C})$
(notably, fuchsian and kleinian groups),
free Burnside groups for large odd exponent,
free soluble groups,
Tarski monsters,
one-relator torsion groups, 
and certain small cancellation groups
(see \cite[p.~10]{LS:77}, \cite{Wu:98}, and \cite{FR:08} for references).




The objective of this paper is the study of profinite CA-groups. The main result of the paper states that such profinite groups possess either pro-$p$ or abelian subgroup of finite index.

\begin{bigthm} \label{t:virt_pronilp_intro}
A profinite CA-group is virtually abelian or virtually pro-$p$. 
\end{bigthm}

The world of pro-$p$ CA-groups is very large and very far from being described.
Kochloukova and the second author introduced in \cite{KZ:11}  a pro-$p$ analogue of limit groups via the operation of extension of centralizers which are pro-$p$ CA-groups. 
The question of description of residually free pro-$p$ CA-groups is still open (it is proved in \cite{KZ:10} that the pro-$p$ completions of surface group are within this class). 
One can see also from results of Wilkes \cite{Wilkes:17} that the pro-$p$ completion of many $3$-manifold groups are CA.  In fact, all these pro-$p$ groups are also CSA-groups.
Thus the study of pro-$p$ CA-groups as well as CSA-groups is a separate quite challenging topic: 
we believe however that the 
description of CA-groups in the class of analytic pro-$p$ groups
is achievable and would be a valuable addition to this topic.  

We prove here that profinite non-abelian CSA-groups are pro-$p$-by-cyclic. Following the standard notation in finite group theory we denote by $O_p(G)$ the intersection of all Sylow pro-$p$ subgroups of the profinite group $G$, \textit{i.e.}, the maximal normal pro-$p$ subgroup of $G$ (open in our situation below).

\begin{bigthm}\label{t:CSA_intro} 
Let $G$ be a profinite CSA-group. Then  one of the following holds:

\begin{enumerate} 

\item[(i)] $G$ is abelian;

\item[(ii)] G is pro-$p$;

\item[(iii)] for some prime $p$ the quotient group  $G/O_p(G)$ is  cyclic of order coprime to~$p-1$.

\end{enumerate}
\end{bigthm}

The next  theorem  tells us that non-prosoluble profinite CA-groups  are  pro-$2$-by-(finite almost simple). We do not know however whether such an example exists.

 Recall also that a group is almost simple if it contains a non-abelian simple group and is contained within the automorphism group of that simple group.

\begin{bigthm} \label{t:simple_intro}
Let $G$ be an infinite profinite CA-group. If $G$ is not prosoluble, then $G/O_2(G)$ is a finite almost simple group.
\end{bigthm}

Our next theorem treats the case of prosoluble virtually pro-$p$ CA-groups. 

\begin{bigthm} \label{t:virt_pro-p_intro}
Let $G$ be an infinite prosoluble CA-group.
If $G$ is virtually pro-$p$, 
then $G$ is a product $O_p(G)K$,
where the (finite) subgroup $K$ is either cyclic or a metacyclic Frobenius group.
\end{bigthm}

Theorems \ref{t:simple_intro} and \ref{t:virt_pro-p_intro} are  proved in \hyperref[s:virt_pro-p]{Section~\ref{s:virt_pro-p}}.

In \hyperref[s:virt_abel]{Section~\ref{s:virt_abel}} we concentrate  on virtually abelian case where we can say much more.

\begin{bigthm} \label{t:virt_abel_intro}
Let $G$ be an infinite profinite CA-group which is virtually abelian.
If~$N$ is the maximal open normal abelian subgroup of $G$,
then $G/N$ is isomorphic to a cyclic group, to a generalized quaternion group, 
or to  $\textup{SL}(2,3)$.
In particular, $G$ is prosoluble.
\end{bigthm}


Throughout the article we require specific results on groups of automorphisms which, for the reader's convenience, are collected in the preliminary \hyperref[s:prelim]{Section~\ref{s:prelim}}. \hyperref[t:virt_pronilp_intro]{Theorem~\ref{t:virt_pronilp_intro}} is proved in \hyperref[s:virt_pronilp]{Section~\ref{s:virt_pronilp}}, where the concept of universal Frattini
cover of profinite groups plays an important role. Then in \hyperref[s:csa]{Section~\ref{s:csa}} we prove \hyperref[t:CSA_intro]{Theorem~\ref{t:CSA_intro}} and conclude the paper with comments and open questions.

\phantomsection
\section{Automorphisms of finite and profinite groups}
\label{s:prelim}
Throughout this paper, every homomorphism of profinite groups is continuous, and every subgroup is closed. We use the standard notation $\mathbb Q$, $\mathbb Q_p$, $\mathbb Z$, $\mathbb Z_p$ for rational,  $p$-adic rational, integer and $p$-adic integer numbers, and $\Phi(G)$ for the Frattini subgroup of a profinite group $G$.  
A pro-$p'$ group is an inverse limit of finite discrete groups whose orders are divisible by primes different from $p$ only.
A cyclic group is always finite.

Many results of the theory of finite groups admit natural interpretation for profinite groups. 
This can be exemplified by the Sylow theorems, the Frattini argument, and so on. 
Throughout the article we use certain profinite versions of facts on finite groups 
without explaining in detail how the results on profinite groups can be deduced 
from the corresponding ones on finite groups. 
On all such occasions the deduction can be performed via the routine inverse limit argument. 

In the present section we concentrate on automorphisms of finite and profinite groups. If $A$ is a group of automorphisms of a group $G$, the subgroup generated by elements of the form $g^{-1}g^\alpha$ with $g\in G$ and $\alpha\in A$ is denoted by $[G,A]$. It is well-known that the subgroup $[G,A]$ is an $A$-invariant normal subgroup in $G$. We also write $C_G(A)$ for the centralizer  of $A$ in $G$ and $A^\#$ for the set of nontrivial elements of $A$.

Our first  lemma is a profinite version of a list of well-known facts on coprime  actions (see for example \cite[Ch.~5 and 6]{Gorenstein:80}).  Here $|G|$ means the order of the profinite group $G$ (see for example \cite[Sec.~2.3]{RZ:00}). A detailed proof of the profinite version of item (iii) can be found in \cite{Shumyatsky:98}.

\begin{lemma}\label{cc}
Let  $A$ be a finite group of automorphisms of a profinite group $G$ such that $(|G|,|A|)=1$. Then
\begin{enumerate}
\item[(i)] $G=C_{G}(A)[G,A]$.
\item[(ii)] $[G,A,A]=[G,A]$. 
\item[(iii)] $C_{G/N}(A)=C_G(A)N/N$ for any $A$-invariant normal subgroup $N$ of $G$.
\item[(iv)] If $G$ is pronilpotent and $A$ is a noncyclic abelian group, then $G=\prod_{a\in A^{\#}}C_{G}(a)$.
\item[(v)]  If  $[G,A]\leq\Phi(G)$, then $A=1$.
\end{enumerate}
\end{lemma}

\begin{lemma}\label{ccc}
Let  $\alpha$ be an automorphism of a finite group $G$ such that $(|G|,|\alpha|)=1$. 
\begin{enumerate}
\item[(i)] If $G$ is cyclic of $2$-power order, then $\alpha=1$.
\item[(ii)] If $G$ is cyclic of prime-power order, then either $\alpha=1$ or $C_G(\alpha)=1$.
\item[(iii)] If $G$ is a generalized quaternion group, then either $\alpha=1$ or $|\alpha|=3$ and $|G|=8$.
\end{enumerate}
\end{lemma}

\begin{proof} 
Let $G$ be a nontrivial cyclic 2-group and let $H$ be the (unique) maximal proper subgroup in $G$. 
Arguing by induction on $|G|$ we can assume that $[H,\alpha]=1$. 
Obviously $\alpha$ induces the trivial automorphism of $G/H$ and so $[G,\alpha,\alpha]=1$ whence $\alpha=1$. 
Thus, (i) is proved.

Assume now that $G$ is cyclic of prime-power order. 
Then $G$ is not a product of proper subgroups. 
Since $G=C_{G}(\alpha)[G,\alpha]$, we either have $G=C_{G}(\alpha)$ or $G=[G,\alpha]$. 
This proves (ii).

Finally, suppose that $G$ is a generalized quaternion group.
If $|G|\geq16$, then $G$ has a (unique) characteristic cyclic subgroup $H$ of index 2. 
As in the proof of item~(i), $H\leq C_G(\alpha)$ and $[G,\alpha,\alpha]=1$ whence $\alpha=1$. 
Of course, if $|G|=8$ and $\alpha\not=1$, then $|\alpha|=3$. 
The proof is now complete.
\end{proof}

Certainly, the next lemma is not new. We include the proof just for the reader's convenience.
\begin{lemma}\label{rankexpo}
If $A$ is a noncyclic group of order $p^2$ acting 
on an additive abelian group $G$,
then $pG\subseteq\sum_{a\in A^\#} C_G(a)$.
\end{lemma}
\begin{proof}  
Each nontrivial element of $A$ belongs to exactly one of the subgroups $A_1,\dots,A_{p+1}$ of order $p$ of $A$. 
For every $i\leq p+1$ and $x\in G$ write $t_i(x)=\sum_{\alpha\in A_i}x^\alpha$ and $t(x)=\sum_{\alpha\in A}x^\alpha$.
It is straightforward that $px=(\sum_it_i(x))-t(x)$.
Since $t(x)\in C_G(A)$ and $t_i(x)\in C_G(A_i)$, we obtain the containment $pG\subseteq\sum_i C_G(A_i)$.
\end{proof}

The following two lemmas were established in \cite{KhMSh:14}.

\begin{lemma}\label{for1}
Suppose that a finite group $G$ admits a Frobenius group of automorphisms
$FH$ with kernel $F$ and complement $H$. If $N$ is an $FH$-invariant normal
subgroup of $G$ such that $C_N(F)=1$, then $C_{G/N}(H)=C_G(H)N/N$.
\end{lemma}

\begin{lemma}\label{for2} Suppose that a finite group $G$ admits a Frobenius group of automorphisms
$FH$ with kernel $F$ and complement $H$ such that 
$C_G(F)=1$. Then $G=\langle C_G(H)^f\ \vert\ {f\in F}\rangle$.
\end{lemma}
Here $\langle X\rangle$ denotes the subgroup generated by $X$. 
From \hyperref[for1]{Lemma~\ref{for1}} and \hyperref[for2]{Lemma~\ref{for2}} we deduce:

\begin{lemma}\label{fro}
Suppose that a profinite group $G$ admits a finite Frobenius group of automorphisms $FH$ 
with kernel $F$ and complement $H$ such that $C_G(F)=1$ and $(|G|,|F|)=1$.
Then $G=\langle C_G(H)^f\ \vert\ f\in F\rangle$.
\end{lemma}
\begin{proof} 
Let $N$ be an $FH$-invariant normal open subgroup of $G$ and $Q=G/N$. 
The group $FH$ naturally acts on $Q$ and \hyperref[cc]{Lemma~\ref{cc}}(iii) states that $C_Q(F)=1$. 
Therefore, by \hyperref[for2]{Lemma~\ref{for2}}, $Q=\langle C_Q(H)^f\ \vert\ {f\in F}\rangle$. 
 \hyperref[for1]{Lemma~\ref{for1}} guarantees that $C_G(H)$ is the inverse limit of centralizers $C_{G/N}(H)$, where $N$ ranges through \mbox{$FH$-invariant} normal open subgroups of $G$ 
and so  $G=\langle C_G(H)^f\ \vert\ f\in F\rangle$, as required.
\end{proof}

We conclude this section by stating explicitly the  profinite version of the result which   is a combination of famous results of  Thompson \cite{Thompson:59} (saying that the group is nilpotent) and Higman  \cite{Higman:57} (bounding the nilpotency class). Its proof can be found in \cite{Shumyatsky:98} for example.

\begin{theorem}\label{higm} Let $p$ be a prime and  $G$ a pro-$p'$ group admitting a fixed-point-free automorphism  of order $p$. Then $G$ is nilpotent of class bounded by some number depending only on $p$.
\end{theorem}

\section{Profinite CA-groups are virtually pronilpotent}
\label{s:virt_pronilp}

The goal of this section is 
to prove \hyperref[t:virt_pronilp_intro]{Theorem~\ref{t:virt_pronilp_intro}},
\textit{i.e.}, that each profinite CA-group is virtually abelian or virtually pro-$p$.


We begin with a fundamental lemma that was inspired by \cite[Thm.~4.10]{ZZ:17}.
Recall first
that an epimorphism of profinite groups
$\widetilde{\varphi}\colon \widetilde{G}\to G$
is a {\it universal Frattini cover} if $\widetilde{G}$ is projective
and $\Ker(\widetilde{\varphi})$ is contained in the Frattini subgroup $\Phi(\widetilde{G})$ of $\widetilde{G}$.
By \cite[Prop. 22.6.1]{FJ}
or  \cite[Lemma 2.8.15]{RZ:00},
{for each profinite group $G$ there exists such an
epimorphism.}

\begin{lemma}\label{l:main}
Let $G$ be a profinite CA-group 
and $\alpha\colon G\longrightarrow Q$ be 
an epimorphism of profinite groups. 
There exists a subgroup $H$ of $G$ such that 
the restriction $\beta=\alpha_{|H}$ is an epimorphism 
and $\Ker(\beta)$ is either 
pro-$p$ or abelian.
\end{lemma}

\begin{proof} 
Let $\widetilde{\varphi} \colon \widetilde{Q}\longrightarrow Q$ be a universal Frattini cover of $Q$.
Since $\widetilde{Q}$ is projective, $\widetilde{\varphi}$ lifts to a 
homomorphism $\omega\colon{\widetilde{Q}\longrightarrow G}$, 
so $\alpha\omega=\widetilde{\varphi}$. 
Put $H=\omega(\widetilde{Q})$. 
Then $\Ker(\beta)$ coincides with $\omega(\Ker(\widetilde{\varphi}))$; 
since the Frattini subgroup $\Phi(\widetilde Q)$ is pronilpotent,
so is $\Ker(\beta)$.
Thus $\Ker(\beta)$ is a direct product of its Sylow subgroups 
hence, being it a CA-group, it is either pro-$p$ or abelian.
\end{proof}

\begin{theorem} \label{c:torsion-free}
A torsion-free profinite CA-group is either pro-$p$ or abelian. 
\end{theorem}

\begin{proof} 
Let $G$ be a torsion-free profinite CA-group. 
By \textit{reductio ad absurdum}, 
let $Q$ be a finite quotient of $G$ that is neither a $p$-group nor abelian. 
Then the subgroup 
$H$ in the statement of \hyperref[l:main]{Lemma~\ref{l:main}}, is either virtually pro-$p$ or virtually abelian. 
Since $H$ is torsion-free and CA, it is either pro-$p$ or abelian;
hence so is $Q$.
\end{proof}

The next lemma provides an important technical tool for dealing with profinite CA-groups.

\begin{lemma}\label{l:normaliz-centraliz} 
Let $p$ be a prime and $A$ an infinite abelian pro-$p'$ subgroup of a profinite CA-group $G$. 
If $A$ normalizes a pro-$p$ subgroup $P$ of $G$, then $PA$ is an abelian subgroup of $G$.
\end{lemma}
\begin{proof} 
Without loss of generality we assume that $G=PA$. 
Note that since $A$ is abelian, $C_G(a)=C_G(A)$ for any nontrivial $a\in A$. 
If $G$ is abelian there is nothing to prove, so we assume that $G$ is not abelian. 
Since $P$ and $A$ are of coprime order, it follows that $G$ is not nilpotent.

Let us show that there exists a bound on the order of torsion elements in $A$. 
Suppose that this is false. 
Let $U$ be an open normal subgroup of $G$; say $U\cap P$ has index $n$ in $P$.
Pick a torsion element $a\in A$ whose order is at least $n!$. 
It is clear that a nontrivial power of $a$, say $a^i$, acts trivially on $P/(U\cap P)$. 
\hyperref[cc]{Lemma~\ref{cc}}(iii) shows that $P=(U\cap P)C_P(a^i)$. 
Hence, $P=(U\cap P)C_P(A)$.
This happens for every open normal subgroup $U$ of $G$.
Therefore  $P=C_P(A)$ and so $G$ is abelian. 
Thus, indeed there exists a bound on the order of torsion elements in $A$.

Suppose now that $A$ contains a noncyclic subgroup $B$ of order $q^2$ for some prime~$q$. 
By \hyperref[cc]{Lemma~\ref{cc}}(iv), $P=\prod_{b\in B^{\#}}C_{P}(b)$ 
and so again $P=C_P(A)$ and $G$ is abelian. 
Therefore without loss of generality we may assume that $A$ has no noncyclic finite subgroups. 
Hence, torsion subgroups in $A$ are cyclic of bounded order. 
It follows that $A$ has an open torsion-free subgroup, say $A'$. 
Since $G$ is CA, it is sufficient to show that $PA'$ is abelian and, 
replacing  $A$ by $A'$ if necessary,  we assume without loss of generality that $A$ is torsion-free.

Since $G$ is non-abelian, it has a finite non-nilpotent quotient $G/N$. 
\hyperref[l:main]{Lemma~\ref{l:main}} shows that $G$ has a subgroup $H$ such that $G=NH$ and $N\cap H$ is pronilpotent. 
Let $A_1$ be a $p'$-Hall subgroup of $H$ and set $P_1=P\cap H$.
Since the image of $H$ in $G/N$ is not nilpotent, we observe that both subgroups $A_1$ and $P_1$ are nontrivial. 
Further, since $A_1$ is torsion-free, it follows that $N\cap A_1\neq1$. 
Set $A_0=A_1\cap N$. 
Since $A_0$ is characteristic subgroup of the pronilpotent group $N\cap H$, 
it is normal in $H$  and moreover $A_0$ normalizes $P_1$. 
We conclude that  $A_0$ centralizes $P_1$ and therefore $P_1\leq C_G(A)$. 
This implies that $H$ is abelian. 
This is a contradiction since the image of $H$ in $G/N$ is not nilpotent.
\end{proof}

Throughout the article we use the celebrated results of Zelmanov \protect{\cite{Zelmanov:92}} 
that every compact torsion group is locally finite and 
every infinite compact group contains an infinite abelian subgroup.
This in particular allows to extend \hyperref[t:finite_CA]{Theorem~\ref{t:finite_CA}} to profinite torsion groups: 
A non-abelian infinite profinite torsion CA-group is a Frobenius group with an abelian kernel of finite exponent and a cyclic complement (see \cite[Sec.~4.6]{RZ:00} for the definition of a Frobenius profinite group).

The \emph{Fitting height} of a prosoluble group $G$ is the length $h(G)$ of a shortest series of normal subgroups all of whose quotients are pronilpotent. The parameter $h(G)$ is finite if, and only if, $G$ is an inverse limit of finite soluble groups of bounded Fitting height.

Also, recall that a Sylow basis in a group $G$ is a pairwise permutable family $\{P_1, P_2, \dots\}$ of $p_i$-Sylow subgroups $P_i$ of $G$, exactly one for each prime. Moreover, the \emph{basis normalizer} of such Sylow basis in $G$ is $\bigcap_i N_G(P_i)$. This subgroup is also known under the name of system normalizer. If $G$ is a profinite group and $T$ is a basis normalizer in $G$, then $T$ is pronilpotent and $G=\gamma_\infty(G)T$, where $\gamma_\infty(G)$ denotes the intersection of the terms of the lower central series of $G$. Furthermore, every prosoluble group $G$ posseses a Sylow basis and any two basis normalizers in $G$ are conjugate (see \cite{Bolker:63} or \cite[Prop.~2.3.9]{RZ:00}). Moreover,  from \cite[Thm.~9.2.7]{Robinson:96} one deduces by the inverse limit argument that if $\gamma_\infty(G)$ is abelian, then $\gamma_\infty(G)$ intersects $T$  trivially.

\begin{lemma}\label{l:Fh3} 
Let $G$ be a prosoluble CA-group. 
If a basis normalizer in $G$ is infinite, then the Fitting height of $G$ is at most 3. 
\end{lemma}
\begin{proof} 
Let $T$ be a basis normalizer in $G$. 
Suppose that $T$ has an infinite $p$-Sylow subgroup. 
Then, by Zelmanov's result, $T$ contains an infinite abelian pro-$p$ subgroup $A$.
If $Q$ is a $q$-Sylow subgroup of $G$ normalized by $A$, 
\hyperref[l:normaliz-centraliz]{Lemma~\ref{l:normaliz-centraliz}} 
shows that $A\subseteq C_G(Q)$ whenever $q\neq p$. 
Choose a $p'$-Hall subgroup $H$ of $G$ normalized by $A$. 
It follows that $H\subseteq C_G(A)$ and so $H$ is abelian. 
Thus, $G$ has an abelian $p'$-Hall subgroup. 
We deduce that $G=O_{p,p',p}(G)$ 
(by the profinite version of 
\cite[Thm.~6.3.2]{Gorenstein:80})
and $h(G)\leq3$. 

Finally, suppose all Sylow subgroups of $T$ are finite. 
Since $T$ is infinite and pronilpotent, we can choose an infinite procyclic subgroup $B$ of $T$. Because the Sylow subgroups of $T$ are finite, $B$ contains a $p'$-subgroup of finite index for each prime $p$.
\hyperref[l:normaliz-centraliz]{Lemma~\ref{l:normaliz-centraliz}} shows that 
if $P$ is any Sylow subgroup of $G$ normalized by $B$,
then $P$ centralizes a subgroup of finite index in $B$ 
and we deduce that $P$ centralizes $B$. Recall that $B$ lies in a basis normalizer in $G$ and hence $B$ normalizes a $p$-Sylow subgroup of $G$ for each prime divisor $p$ of the order of $G$. Therefore $B\subseteq Z(G)$ and $G$ must be abelian.
\end{proof}

\begin{lemma}\label{fitting} 
The Fitting height of a prosoluble CA-group is at most~5.
\end{lemma}
\begin{proof} 
Let $G$ be a prosoluble CA-group.
In view of \hyperref[l:Fh3]{Lemma~\ref{l:Fh3}}
we can assume that $G$ has a finite basis normalizer $T_1$. 
Let $\{P_1, P_2, \dots\}$ be a Sylow basis of $G$ such that $T_1$ normalizes each Sylow subgroup $P_i$. 
Let $K_1=\gamma_\infty(G)$. 
We have $G=K_1T_1$. 
Let $T_2$ be the basis normalizer of
the Sylow basis $\{P_1\cap K_1,P_2\cap K_1, \dots\}$ in $K_1$. 
If $T_2$ is infinite, then, by  \hyperref[l:Fh3]{Lemma~\ref{l:Fh3}}, $h(K_1)\leq3$ and $h(G)\leq4$. 
We therefore can assume that $T_2$ is finite. 
The specific choice of the Sylow basis $\{P_1\cap K_1, P_2\cap K_1, \dots\}$ 
guarantees that $T_1$ normalizes $T_2$. 
Note that $G=K_2T_2T_1$, where $K_2=\gamma_\infty(K_1)$. 
It follows that the subgroup $T_2T_1$ is not nilpotent since $G/K_2$ is not.
Let $T_3$ be the basis normalizer of
the Sylow basis $\{{P_1\cap K_2}, {P_2\cap K_2}, \dots\}$ in $K_2$. 
If $T_3$ is infinite, then $h(K_2)\leq3$ and $h(G)\leq5$. 

We therefore can assume that $T_3$ is finite. 
The specific choice of the Sylow basis $\{{P_1\cap K_2}, {P_2\cap K_2}, \dots\}$ 
guarantees that $T_2T_1$ normalizes $T_3$. 
Note that $G=K_3T_3T_2T_1$, where $K_3=\gamma_\infty(K_2)$. 
It follows that $h(T_3T_2T_1)\geq3$ since $h(G/K_3)=3$.
But $T_3T_2T_1$ is a  finite soluble CA-group and so by \hyperref[t:finite_CA]{Theorem~\ref{t:finite_CA}} it has Fitting height at most 2, a contradition.
Thus the case where $T_3$ is finite  does not occur and the lemma is proved.
\end{proof}

We will require the following lemma.
\begin{lemma}\label{tata} 
Let $N$ be a normal subgroup of a finite group $G$ such that
$G/N$ is soluble. 
The group $G$ has a soluble subgroup $H$ such that $G=NH$.
\end{lemma}
\begin{proof} 
Choose a $2$-Sylow subgroup $S$ of $N$ and write $H=N_G(S)$. 
By the Frattini argument $G=NH$ 
and by the theorem of Feit and Thompson $H\cap N$ is soluble. 
It follows that $H$ is soluble, as required.
\end{proof}

Denote by $\mathfrak C$ the class of all finite groups 
whose soluble subgroups are of Fitting height at most 5.
Obviously, $\mathfrak C$ is closed under taking subgroups.

\begin{lemma}\label{l:quotient_closed} 
The class $\mathfrak C$ is closed under taking quotients.
\end{lemma}
\begin{proof} 
Let $G\in\mathfrak C$. 
It is sufficient to show that if $N$ is a normal subgroup of $G$ such that $G/N$ is soluble, 
then $h(G/N)\leq5$. 
By \hyperref[l:quotient_closed]{Lemma~\ref{tata}} we have $G=NH$ where $H$ is a soluble subgroup and 
so $G/N$ is isomorphic to a quotient of $H$. 
Hence, $h(G/N)\leq5$, as required.
\end{proof}

\begin{lemma} 
Any profinite CA-group is pro-$\mathfrak C$.
\end{lemma}
\begin{proof} 
Let $G$ be a profinite CA-group. 
It is sufficient to show that 
if $K$ is a soluble subgroup of a finite homomorphic image of $G$,
then $h(K)\leq5$. 
By \hyperref[l:main]{Lemma~\ref{l:main}} (or by a profinite version of  \hyperref[l:quotient_closed]{Lemma~\ref{tata}}) there exists a prosoluble subgroup $H$ in $G$ that maps onto $K$. 
 \hyperref[fitting]{Lemma~\ref{fitting}} states that $h(H)\leq 5$ whence $h(K)\leq5$.
\end{proof}

Recall that the nonsoluble length $\lambda(G)$ of a finite group $G$ is defined 
as the minimum number of nonsoluble factors in a normal series each of whose factors is
either soluble or a nonempty direct product of non-abelian simple groups. Khukhro and the first author \cite[Cor.~1.2]{KhSh:15} proved that 
the nonsoluble length of a finite group $G$ does not exceed the maximum Fitting height of soluble subgroups of $G$. It follows that 
for any group $G$ in $\mathfrak C$ we have $\lambda(G)\leq5$. 
Combining this with \hyperref[l:quotient_closed]{Lemma~\ref{l:quotient_closed}} we conclude that each group $G$ in $\mathfrak C$ 
has a characteristic series of length at most 35 each of whose factors is either nilpotent 
or a direct product of non-abelian simple groups. 
More precisely, each group $G$ in $\mathfrak C$ has a characteristic series
\[1=G_0\leq G_1\leq\dots\leq G_{35}=G\]
such that 
the factors $G_6/G_5$, $G_{12}/G_{11}$, $G_{18}/G_{17}$, $G_{24}/G_{23}$, $G_{30}/G_{29}$ 
are direct product of non-abelian simple groups while the other factors are nilpotent. 
Important results of Wilson \cite[Lemma~2 and Lemma~3]{Wilson:83} now guarantee that any {pro-$\mathfrak C$} group has a characteristic series of length at most 35 each of whose factors is either pronilpotent or a Cartesian product of non-abelian simple groups. 
Thus, we have proved the following result.

\begin{lemma}\label{series} 
Any profinite CA-group has a characteristic series of length at most 35 
each of whose factors is either pronilpotent 
or a Cartesian product of non-abelian finite simple groups.
\hfill \qedsymbol
\end{lemma}

We remark that our final results show that the characteristic series in the above lemma 
can be chosen of length at most 3 with the first term being an open subgroup.

In what follows $\pi(G)$ denotes the set of prime divisors of the order of a profinite group $G$.

\begin{lemma}\label{l:metapronilp} 
Let $G$ be a profinite CA-group having a normal subgroup $N$ such that
both $N$ and $G/N$ are pronilpotent. Then $G$ is virtually pronilpotent.
\end{lemma}
\begin{proof} 
Suppose first that $N$ is abelian 
and arguing by contradiction assume that $G$ is not virtually nilpotent. 
Let $K=\gamma_\infty(G)$ and let $T$ be a basis normalizer in $G$. 
So $K$ is abelian and thus, by a profinite version of the corresponding result on finite groups 
(see \cite[Thm.~9.2.7]{Robinson:96}), we have $G=KT$ and $K\cap T=1$. 
Without loss of generality we assume that $T$ is infinite. 
Further, assume that $G$ is not virtually pro-$p$ 
and so we choose $p\in\pi(K)$ such that $T$ contains an infinite pro-$p'$ subgroup. 
Since $T$ is pronilpotent, we remark that either $T$ is abelian or pro-$q$ for a prime $q\neq p$.
By \hyperref[l:normaliz-centraliz]{Lemma~\ref{l:normaliz-centraliz}},
any infinite abelian pro-$p'$ subgroup $A$ of $T$ centralizes a $p$-Sylow subgroup $P$ of $K$.
If $T$ is abelian, we deduce that $A\subseteq Z(G)$; whence the whole group $G$ is abelian.
We therefore assume that $T$ is not abelian and so $T$ is a pro-$q$ group.
Since infinite abelian pro-$q$ subgroups of $T$ centralize $P$, 
it follows that all infinite procyclic subgroups of $T$ centralize $P$ and so $T/C_T(P)$ is torsion. 
In view of Zelmanov's theorem we conclude that $T/C_T(P)$ is locally finite.
Let $C=C_G(P)$. Of course, $C$ is abelian. 
If $T/C_T(P)$ is cyclic or generalized quaternion, then $C$ is open in $G$; 
thus $G$ is virtually abelian. 
So suppose that $T/C_T(P)$ is neither cyclic nor generalized quaternion.
Then, by \cite[Thm.~5.4.10(ii)]{Gorenstein:80}, $T/C_T(P)$ contains a noncyclic subgroup $B$ of order $q^2$.
In view of \hyperref[cc]{Lemma~\ref{cc}}(iv), we have $P=\prod_{b\in B^\#} C_P(b)$.
Thus, for at least one nontrivial element $b$ in $B$ we obtain $C_P(b)\neq1$, 
and so $b\in C$; a contradiction. 
This proves the lemma in the case where $N$ is abelian.

Assume now that $N$ is non-abelian and so $N$ is a pro-$p$ group for a prime $p$.
If $G/N$ is virtually pro-$p$, there is nothing to prove. 
Therefore we assume that the pronilpotent group $G/N$ contains an infinite pro-$p'$ subgroup. 
It follows that $G$ contains an infinite abelian pro-$p'$ subgroup which, 
by \hyperref[l:normaliz-centraliz]{Lemma~\ref{l:normaliz-centraliz}}, centralizes $N$.
Since $C_G(N)\neq1$, it follows that $N$ is abelian.
This case was treated in the previous paragraph.
\end{proof}

We are ready to prove the result that entitles this section.

\begin{customthm}{\ref{t:virt_pronilp_intro}}
A profinite CA-group is virtually abelian or virtually pro-$p$. 
\end{customthm}

\begin{proof} 
We will actually prove the equivalent statement that a profinite CA-group $G$ is virtually pronilpotent. 
By \hyperref[series]{Lemma~\ref{series}}, $G$ has a normal series of length at most~35
\[1= G_0\leq\dots\leq G_l=G\] 
such that each factor $G_{i+1}/G_i$ is either pronilpotent 
or a Cartesian product of non-abelian finite simple groups.
Let $l$ be the minimum of lengths of such series. 
If $l=1$, then $G$ is either pronilpotent or finite.
We therefore assume that $l\geq2$ and use induction on $l$.
By induction, $G_{l-1}$ is virtually pronilpotent.
{Set $R=G_{l-1}$ and 
let $N$ be the maximal open normal pronilpotent subgroup in $R$.}
Set $H={\{x\in G\ \vert\ [R,x]\leq N\}}$.
Since $R/N$ is finite, it is clear that $H$ is an open subgroup in $G$.
Note that the image of $R\cap H$ in $H/N$ is central.
It is sufficient to show that $H$ is virtually pronilpotent.
Thus, without loss of generality we can assume that $R/N$ is central in $G/N$.

Suppose first that $G/R$ is pronilpotent.
Since $R/N$ is central in $G/N$, it follows that $G/N$ is pronilpotent 
and the theorem is immediate from \hyperref[l:metapronilp]{Lemma~\ref{l:metapronilp}}.

Therefore it remains to deal with the case 
where $G/R$ is a Cartesian product of finite non-abelian simple groups.
In that case $G$ is a product $\prod_{i\in I} S_i$ of normal subgroups $S_i$
such that $S_i/R$ is a simple direct factor of $G/R$.

Assume that the set of indices $I$ is infinite. Let $P$ be a $p$-Sylow subgroup of the pronilpotent group $N$.
For each $i\in I$ choose a $p'$-element $a_i\in S_i - R$.
Note that the image in $G/N$ of the subgroup $A=\langle a_i\ \vert\ i\in I\rangle$
is infinite and nilpotent of class at most two.
By \hyperref[l:metapronilp]{Lemma~\ref{l:metapronilp}}, $NA$ is virtually pronilpotent.
Let $M$ be the maximal normal open pronilpotent subgroup of $NA$.
Clearly, $N\subseteq M$. 
Since $M$ is open, we can choose two distinct indexes $i$ and $j$ in $I$ such that $a_ia_j^{-1}\in M$.
Moreover, $a_ia_j^{-1}$ is a $p'$-element modulo $N$.
Therefore the $p'$-Hall subgroup of the procyclic subgroup $\langle a_ia_j^{-1}\rangle$ is nontrivial.
Choose a generator $b$ of the $p'$-Hall subgroup of $\langle a_ia_j^{-1}\rangle$.
We see that $b\in S_iS_j - R$. Further, taking into account that $b$ is a $p'$-element in $M$ we deduce that $b$ centralizes $P$.
Since $P$ is normal in $G$, the subgroup $[G,b]$ centralizes $P$ and we conclude that $[G,b]$ is abelian.
Note that any normal subgroup of a Cartesian product of non-abelian finite simple group is a product of some of the direct factors.
Observe also that in our case the subgroup $[G,b]$ is normal in $G$ and therefore the subgroup $R[G,b]$ contains both $S_i$ and $S_j$.
This is a contradiction since the subgroups $S_i$ and $S_j$ are not abelian modulo $R$.
Thus, the set $I$ cannot be infinite.
The proof is complete.
\end{proof}

\section{Virtually abelian CA-groups}
\label{s:virt_abel}

It is clear that
a Frobenius profinite group is a CA-group if, and only if,
its Frobenius kernel is abelian and its Frobenius complement is cyclic.
So, for instance, an 
infinite profinite CA-group which is virtually abelian
can have a cyclic quotient modulo a normal abelian subgroup.

In the present section we show that 
infinite profinite CA-group which are virtually abelian
can be categorized according to their quotients modulo maximal abelian normal subgroups.
Besides the cyclic groups, there are only two more possibilities for the quotients: 
the generalized quaternion groups $\mathbf{Q}_{2^{n}}$ of order $2^{n}$, 
or the special linear group $\textup{SL}(2,3)$. Recall that the group $\textup{SL}(2,3)$ has order 24 and is isomorphic to a semidirect product of the quaternion group $\mathbf{Q}_{8}$ by the group of order 3. We conclude the section constructing an example of a CA-group with the quotient isomorphic to $\textup{SL}(2,3)$.

Now, 
let $G$ be an infinite profinite CA-group which is virtually abelian and 
let $N$ be the (unique) maximal open normal abelian subgroup in $G$. 
\emph{Let $\bar{G}=G/N$.}
The group $\bar{G}$ acts faithfully on $N$ by automorphisms and $C_N(\phi)=1$ for each $\phi$ in $\bar{G}$. It follows that every element in $G - N$ has finite order. We claim that the Sylow subgroups of $\bar{G}$ are either cyclic or generalized quaternion.
To see this, suppose that $\bar{G}$ contains a noncyclic subgroup $\bar{A}$ of order $p^2$ for some prime $p$.
\hyperref[rankexpo]{Lemma~\ref{rankexpo}} implies that $N$ is of exponent $p$, that is $N^p=1$.
Choose a nontrivial element $x$ in $N$ and consider the action of $\bar{A}$ on the finite $p$-group $\langle x^{\bar{A}}\rangle$.
Since the semidirect product of $\langle x^{\bar{A}}\rangle$ by $\bar{A}$ is a finite $p$-group,
we easily see that $\bar{A}$ has nontrivial centralizer in $\langle x^{\bar{A}}\rangle$.
This is a contradiction since $C_N(\phi)=1$ for each $\phi$ in $\bar{G}$.
Hence, $\bar{G}$ has no noncyclic subgroups of order $p^2$ and so indeed, by \cite[Thm.~5.4.10(ii)]{Gorenstein:80}, the Sylow subgroups of $\bar{G}$ are either cyclic or generalized quaternion.

\begin{lemma} \label{l:fitt}
With the above assumptions suppose that $\bar{G}$ is nilpotent.
Then $\bar{G}$ is either cyclic or generalized quaternion.
\end{lemma}
\begin{proof} 
The result is obviously correct if $2\not\in\pi(\bar{G})$;
so we assume that $2\in\pi(\bar{G})$.
Let $S$ be a $2$-Sylow subgroup of $G$ and $\bar{S}$ the image of $S$ in $\bar{G}$.
We only need to show that if $\bar{S}$ is generalized quaternion, then $\bar{G}=\bar{S}$.
Suppose that this is false and set $K=NS$. 
By the Frattini argument, $G=KN_G(S)$. 
Since $\bar{G}\neq\bar{S}$, we can choose an odd order element $a$ in $N_G(S)$ such that $a\not\in N$. 
Since $C_{\bar{S}}(a)$ is non-abelian (of course, $C_{\bar{S}}(a)$ is precisely $\bar{S}$), 
it follows from \hyperref[cc]{Lemma~\ref{cc}}(iii) that $C_S(a)$ is non-abelian, too.
This is a contradiction.
\end{proof}

\begin{lemma}\label{l:pq}
With the above assumptions suppose
that $\bar{G}$ contains a metacyclic subgroup $\bar{H}$ whose order is divisible by two different primes.
Then $\bar{H}$ is cyclic.
\end{lemma}
\begin{proof} 
Without loss of generality we can assume that $\bar{G}=\bar{H}$.
Choose a counterexample with $|\bar{G}|$ as small as possible.
Then $\bar{G}$ is a non-abelian semidirect product of its normal $p$-Sylow subgroup $\bar{P}$ by a $q$-Sylow subgroup $\bar{Q}$.
Moreover, since our counterexample is minimal, 
taking into account \hyperref[ccc]{Lemma~\ref{ccc}}(ii) we deduce that $\bar{P}$ has order exactly $p$.
Then $\bar{P}=[\bar{P},\bar{Q}]$.
Let $\bar{a}$ be an element of order $q$ in $\bar{Q}$ and set $\bar{Q}_1=\langle a\rangle$.
Either the group $\bar{P}\bar{Q}_1$ is Frobenius or $[\bar{P},\bar{a}]=1$.
We denote by $P$ and $Q$ some $p$-Sylow and $q$-Sylow subgroups of $G$ which map onto $\bar{P}$ and $\bar{Q}$, respectively.

Suppose first that $\bar{P}\bar{Q}_1$ is a Frobenius group.
Assume additionally that $Q_0=Q\cap N$ is nontrivial.
Consider the natural action of $\bar{P}\bar{Q}_1$ on $Q_0$. Then
\hyperref[fro]{Lemma~\ref{fro}} shows that $C_{Q_0}(\bar{Q}_1)\neq1$.
This is a contradiction. Thus, $Q_0=1$ and $q\not\in\pi(N)$.
Remark that $\bar{a}$ induces a fixed-point-free automorphism of prime order $q$ of $NP$.
\hyperref[higm]{Theorem~\ref{higm}} now states that $NP$ is nilpotent.
Since $NP$ is CA and $N$ is the maximal abelian subgroup,
we conclude that $P\subseteq N$, a contradiction again.
Thus, $\bar{P}\bar{Q}_1$ cannot be a Frobenius group. 

It follows that $[\bar{P},\bar{a}]=1$. 
Let $b\in Q$ be an element such that $a\in\langle b\rangle$ and $\bar{b}$ generates $\bar{Q}$. 
Note that $NP$ is normal in $G$ and therefore by the Frattini argument $G=NN_G(P)$. 
Thus $b$ can be written as a product $xb_1$, where $x\in N$ and $b_1\in N_G(P)$. 
Here the order of $b_1$ equals that of $b$. 
So replacing if necessary $b$ by $b_1$ and $a$ by an element of order $q$ in $\langle b_1\rangle$ we can assume that $b$ normalizes $P$.

Since $[\bar{P},\bar{a}]=1$, we  deduce  from \hyperref[cc]{Lemma~\ref{cc}}(iii) that $C_P(a)$ is not contained in $N$. 
Since $G$ is CA and since $\bar{P}$ is of order $p$,
it follows that $C_P(b)$ is not contained in $N$ and so $P=P_0C_P(b)$, where $P_0=P\cap N$.
In this case $\bar{P}$ and $\bar{Q}$ commute, whence  $G/N$ is abelian. 
The proof is complete.
\end{proof}

In what follows we write $F(K)$ for the Fitting subgroup of a finite group $K$. 
We are now ready to prove \hyperref[t:virt_abel_intro]{Theorem~\ref{t:virt_abel_intro}}.
We restate it for the reader's convenience.

\begin{customthm}{\ref{t:virt_abel_intro}}
Let $G$ be an infinite profinite CA-group which is virtually abelian.
If~$N$ is the maximal open normal abelian subgroup of $G$,
then $G/N$ is isomorphic to a cyclic group, to a generalized quaternion group, 
or to  $\textup{SL}(2,3)$.
In particular, $G$ is prosoluble.
\end{customthm} 

\begin{proof} 
Let $\bar G=G/N$.
From \hyperref[l:fitt]{Lemma~\ref{l:fitt}} we know that for any odd prime $p\in\pi(\bar{G})$  the $p$-Sylow subgroup $\bar P$ of $\bar{G}$ is cyclic.
Together with \hyperref[l:pq]{Lemma~\ref{l:pq}} and Frobenius' normal $p$-complement theorem (\cite[Thm.~7.4.5]{Gorenstein:80}) 
this implies that $\bar{P}$ has a normal $p$-complement for any odd $p\in\pi(\bar{G})$.
We conclude that $\bar{G}$ is soluble.
Hence, $F(\bar{G})\neq1$ and moreover $F(\bar{G})$ either is cyclic or generalized quaternion. Recall that in a finite soluble group the Fitting subgroup contains its centralizer \cite[Thm.~6.1.3]{Gorenstein:80}.
If $F(\bar{G})$ is cyclic, then $\bar G$ is either cyclic or contains a Frobenius subgroup so \hyperref[l:pq]{Lemma~\ref{l:pq}} shows that $\bar{G}$ is cyclic.
Therefore we only need to deal with the case where $F(\bar{G})$ is a generalized quaternion.
If $F(\bar{G})$ is a generalized quaternion of order at least 16, then
\hyperref[ccc]{Lemma~\ref{ccc}}(iii) shows that $F(\bar{G})=\bar{G}$ and we are done.

Hence, without loss of generality assume that $F(\bar{G})$ is the quaternion group of order 8
and $\bar{G}$ has an element of order 3, that acts on $F(\bar{G})$ non-trivially.
Arguing as above we deduce from Frobenius' theorem that $\bar{G}$ has a normal 3-complement
whence $\bar{G}$ is isomorphic to the group $\textup{SL}(2,3)$.
\end{proof}


 Now we produce an example of a profinite CA-group for which the quotient over a normal abelian subgroup is isomorphic to  $\textup{SL}(2,3)$.
 
\begin{lemma}\label{bbb} 
The group $\textup{GL}(4,\mathbb{Z})$ contains a subgroup $S$ isomorphic to $\textup{SL}(2,3)$ 
whose nonidentity elements fix no nonzero element of $\mathbb{Z}^4$.
\end{lemma}

\begin{proof} 
Let $K=\mathbb{Q}(\zeta)$, where $\zeta$ is a primitive third root of unity; 
in particular $\zeta^2+\zeta+1=0$.
Set $a=\left(\begin{smallmatrix}0&\phantom{^2}1\\ -1&\phantom{^2}0\end{smallmatrix}\right)$, $d=\left(\begin{smallmatrix}\zeta&\phantom{^2}1\\0&\phantom{^2}\zeta^2\end{smallmatrix}\right)$, and $b:=dad^{-1}$.
Then $a^2=-\textup{Id}=b^2$ and $d$ has order $3$.
We have 
$$d^2=d^{-1}=\begin{pmatrix}\zeta^2&-1\\0&\zeta\end{pmatrix}\quad \textup{ and} \qquad b=\begin{pmatrix}-\zeta^2&-\zeta\\-\zeta&\zeta^2\end{pmatrix}.$$
In addition, $b^{-1}ab=a^{-1}$; so  $\langle a,b\rangle$ is isomorphic to $\mathbf{Q}_8$. 

Moreover, $d^{-1}ad 
= \left(\begin{smallmatrix}\zeta&-\zeta^2\\-\zeta^2&-\zeta\end{smallmatrix}\right)=ba$.
Therefore, $d$ normalizes $\langle a,b\rangle$.
It follows that $\langle a,d\rangle\cong \textup{SL}(2,3)$ 
(see \cite[Sec.~8.2]{Rose:09}). 

Note that $[K:\mathbb{Q}]=2$ and $1,\zeta$ (or $\zeta,\zeta^2$) can be chosen as a basis of $K$ as a vector space over $\mathbb{Q}$.
Therefore, the $2$-dimensional vector space $K^2$ over $K$ can be viewed as a $4$-dimensional vector space over $\mathbb{Q}$.
It follows that the mapping $K\rightarrow \mathbb{Q}^2$ yields an injective mapping $\langle a,g\rangle\rightarrow \textup{GL}(4,\mathbb{Q})$.
Denote by $S$ the image of $\langle a,g\rangle$ in $\textup{GL}(4,\mathbb{Q})$.
Then $S\cong \textup{SL}(2,3)$.
By \cite[Thm.~73.6]{CR:62}, 
any representation of a finite group over $\mathbb{Q}$ is equivalent to a representation over $\mathbb{Z}$, 
so we get a representation  $\textup{SL}(2,3)\rightarrow \textup{GL}(4,\mathbb{Z})$, 
that is, we can assume that  $S\subseteq \textup{GL}(4,\mathbb{Z})$. 

It is rather obvious that every nonidentity element of $\langle a,g\rangle$ has no eigenvalue $1$, 
and hence acts fixed-point-freely on $K^2$.
As the mapping $K^2\rightarrow \mathbb{Q}^4$ above is an isomorphism of  $\mathbb{Q}$-vector spaces,
one observes that  every nonidentity element of $S$ has no eigenvalue $1$ on $\mathbb{Q}^4$, 
and hence on $\mathbb{Z}^4$.
So the lemma follows. 
\end{proof}

Note that $\textup{SL}(2,3)$ has no faithful representation into $\textup{GL}(2,\mathbb{Q})$, 
as already $\mathbf{Q}_8$ is not a subgroup of $\textup{GL}(2,\mathbb{Q})$, see 
\cite{Serre:08}, where Serre provides a criterion 
for $\mathbf{Q}_8$ to be a subgroup of  $\textup{GL}(2,K)$ and $\textup{GL}(2,R)$, 
where $K$ is an arbitrary finite extension of $\mathbb{Q}$ and $R$ is the ring of integers of $K$.

\begin{lemma}\label{aaa}
Let $G$ be a pro-$p$ group that is a semidirect product of 
a nontrivial normal subgroup $M$ by a $k$-generator group $A$.
Then $G$ is at least $(k+1)$-generated.
\end{lemma}
\begin{proof} 
Note that $M\neq[M,A]$.
Since $[M,A]\subseteq\Phi(G)$, we pass to the quotient $G/[M,A]$ to see that 
it is just the direct product of $M/[M,A]$ and $A$, whence the result is immediate.
\end{proof}

\begin{proposition}
There exists a profinite CA-group $G$ having an abelian normal subgroup $N$ 
such that $G/N$ is isomorphic to $\textup{SL}(2,3)$.
\end{proposition}

\begin{proof} 
Let $S=\textup{SL}(2,3)=\langle a,d\rangle$, where $a$ is of order 4 and $d$ of order 3.
The $2$-Sylow subgroup in $S$ is $Q=\langle a,a^d\rangle$ and $a^2$ is central in $S$.

Let $V$ be the $4$-dimensional free $\mathbb{Z}_2$-module.
By \hyperref[bbb]{Lemma~\ref{bbb}}, $S$ can be embedded into $\textup{GL}(4,\mathbb{Z})$
and hence in $\textup{GL}(4,\mathbb{Z}_2)=GL(V)$ in such a way that $C_V(x)=0$ for each nontrivial $x\in S$.  To see this simply observe that 1 is not an eigenvalue for $x$.

Of course $a^2$ acts on $V$ taking each vector $v\in V$ to $-v$. 
Let $H$ be the natural semidirect product of $V$ by $S$.
Clearly, $C_H(d)=\langle da^2\rangle$. 
Note that if $h\in H$ and $C_H(h)$ is non-abelian, then necessarily $C_H(h)$ has a subgroup isomorphic to $\mathbf{Q}_8$. 
Choose any nonzero vector $v\in V$. 
Consider the 2-generated subgroup $G=\langle va,d\rangle$ and set $N=G\cap V$. 
We see that $G/N\cong \textup{SL}(2,3)$.
Further $N\neq\{0\}$ because otherwise $d$ would centralize $(va)^2$ which would lead to a contradiction.
Note that $a^2$ is contained in $G$ because otherwise $d$ would act fixed-point-freely on the $2$-Sylow  subgroup $P$ of $G$.
By \hyperref[higm]{Theorem~\ref{higm}} that would imply that $P$ is nilpotent; 
a contradiction with the assumption that $C_V(x)=0$ for each $x\in S$. 

Suppose that $G$ is not CA.
Then $G$ contains an element whose centralizer has a subgroup $B$ isomorphic to $\mathbf{Q}_8$.
It follows that $P$ is the semidirect product $NB$.
By \hyperref[aaa]{Lemma~\ref{aaa}} the subgroup $P$ has at least 3 generators.
Let us have a closer look at the action of $d$ on $\bar{P}=P/P^2$.
Since $\bar{P}\langle d\rangle$ is a quotient of $G$, the group $\bar{P}\langle d\rangle$ is 2-generator
and so $\bar{P}$ cannot be of order greater than 8.
This is because, since $d$ is of order 3, any $d$-orbit of an element of $\bar{P}$ generates a subgroup of order at most 8.
Taking into account that $P$ has at least 3 generators (\hyperref[aaa]{Lemma~\ref{aaa}}) we conclude that the order of $\bar{P}$ is precisely 8. 
Obviously, groups of order 8 do not admit fixed-point-free automorphisms of order 3 so necessarily $C_{\bar{P}}(d)\neq1$. 
Since $C_H(d)=\langle da^2\rangle$, we deduce that $a^2\not\in P^2$ and therefore $d$ acts fixed-point-freely on $P^2$. 
Hence, by \hyperref[higm]{Theorem~\ref{higm}}, $P^2$ is nilpotent. 
We see that $B^2$ has nontrivial centralizer in $V$. 
Clearly, this implies that $a^2$ has nontrivial centralizer in $V$, a contradiction.

Thus, the group $G$ is CA with $G/N\cong \textup{SL}(2,3)$.\end{proof}

\section{Virtually pro-\texorpdfstring{$p$}{p} CA-groups}
\label{s:virt_pro-p}

We turn now to the structure of profinite CA-groups which are virtually pro-$p$. 
Theorems \ref{t:virt_pro-p_intro} and \ref{t:simple_intro}
will be proved in this section.

The following lemma is almost obvious.

\begin{lemma}\label{v4}
Let $G$ be an infinite non-abelian virtually pro-$p$ CA-group. 
Any $q$-Sylow subgroup of $G$ for a prime $q\neq p$ is cyclic.
\end{lemma}
\begin{proof} 
Let $Q$ be a $q$-Sylow subgroup of $G$.
By \hyperref[cc]{Lemma~\ref{cc}}(iv), all abelian subgroups of $Q$ are cyclic.
We deduce that either $Q$ is cyclic or is generalized quaternion.
Since the quaternion group is nilpotent and non-abelian, the lemma follows.
\end{proof}

It is well-known that non-abelian finite simple groups have noncyclic $2$-Sylow subgroups \cite[Thm.~7.6.1]{Gorenstein:80}. 
Therefore we deduce from \hyperref[v4]{Lemma~\ref{v4}} the following result:

\begin{proposition}\label{p:prosoluble}
Let $p$ be an odd prime.
An infinite virtually pro-$p$ CA-group is prosoluble. \hfill \qedsymbol
\end{proposition}

We are now in position to complete the proof of \hyperref[t:virt_pro-p_intro]{Theorem~\ref{t:virt_pro-p_intro}}, which we restate here for the reader's convenience.

\begin{customthm}{\ref{t:virt_pro-p_intro}} 
Let $G$ be an infinite prosoluble CA-group.
If $G$ is virtually pro-$p$, 
then $G$ is a product $O_p(G)K$,
where the (finite) subgroup $K$ is either cyclic or a metacyclic Frobenius group.
\end{customthm}

\begin{proof} 
Let $P=O_p(G)$ and $T=O_{p,p'}(G)$. 
Let $H$ be a Hall $p'$-subgroup of $T$. 
By the Frattini argument $G=PN_G(H)$.
Note that $N_G(H)\cap P=C_P(H)$.
Hence, if $N_G(H)$ is infinite, $N_G(H)$ is virtually abelian.
By \hyperref[t:virt_abel_intro]{Theorem~\ref{t:virt_abel_intro}}, 
$N_G(H)$ is a semidirect product of an open abelian subgroup by a cyclic subgroup.
Therefore there exists a cyclic $p$-subgroup $B$ such that $N_G(H)=C_P(H)HB$, where $C_P(H)H$ is abelian.
It follows that $G=PHB$. Of course, $HB$ is either cyclic or  a Frobenius group.

Suppose now $N_G(H)$ is finite. 
By the classification of finite CA-groups (\hyperref[t:finite_CA]{Theorem~\ref{t:finite_CA}}),
$N_G(H)$ is a Frobenius group with abelian kernel and cyclic complement.
The result follows.
\end{proof}

Next we give an example of a profinite CA-group which is neither pro-$p$ nor virtually abelian. 
We use  \cite[Lemma~3.9]{FGRS:16} to show that $\textup{PSL}(2,\mathbb{Q}_p)$ is a CA-group if, and only if, $p=2$.

\begin{lemma}\label{squares}
Let $p$ be a prime. 
The number $-1$ is a sum of two squares over $\mathbb{Q}_p$ if, and only if, $p>2$.
\end{lemma}

\begin{proof} 
Consider the equation $x^2+y^2=-1$.
If $p>2$, take a solution over $\mathbb{Z}/p\mathbb{Z}$
(since there are $(p+1)/2$ squares therein)
and use Hensel's lemma to lift it to $\mathbb{Z}_p$.
If $p=2$, note that there is no solution modulo $4$;
so $x$ or $y$ cannot be in $\mathbb{Z}_2$.
Write $x=2^m u$ and $y=2^n v$ where $u$ and $v$ are units in $\mathbb{Z}_2$
and $m$ and $n$ are integers, or write $v=0$ if $y=0$.
It follows from the equation that $m=n<0$; so $u^2+v^2=-4^m$.
Since the square of any unit is congruent to $1$ modulo $4$,
we get a contradiction.
\end{proof}

\begin{proposition}
The group $\textup{PSL}(2,\mathbb{Z}_p)$ is CA if, and only if, $p=2$.
\end{proposition}

\begin{proof} 
By  \cite[Lemma~3.9]{FGRS:16}) the group $\textup{PSL}(2,K)$ over a field $K$ of characteristic 0 is CA if and only if $-1$ is the sum of two squares of elements of $K$. Therefore, by \hyperref[squares]{Lemma~\ref{squares}}, $\textup{PSL}(2,\mathbb{Q}_p)$ is CA if and only if $p=2$. Thus $\textup{PSL}(2,\mathbb{Z}_2)$ is CA as a subgroup of $\textup{PSL}(2,\mathbb{Q}_2)$. 
 
Let $p$ be an odd prime.
For $p\neq 5$ consider the following matrices of $\textup{SL}(2,\mathbb{Z}_p)$
\[a=\begin{pmatrix}-x&y\\ y&x\end{pmatrix}, \qquad
b=\begin{pmatrix}0&1\\ -1&0\end{pmatrix}, \qquad
c=\begin{pmatrix}\frac35&\frac45\\ -\frac45&\frac35\end{pmatrix} \, .\]
One can check that modulo the center
$a$ commutes with $b$, $b$ commutes with $c$, but $a$ does not commute with $c$; 
so $\textup{PSL}(2,\mathbb{Z}_p)$ is not CA.
For $p=5$, one can replace $(3,4,5)$ with $(5,12,13)$ in the matrix $c$ and use the same argument to conclude that $\textup{PSL}(2,\mathbb{Z}_5)$ is not CA.
\end{proof}

Note that the obvious epimorphism $\textup{PSL}(2,\mathbb{Z}_2)\longrightarrow \textup{PSL}(2,2)$ 
has a nonabelian pro-$2$ group as its kernel and the Frobenius group $S_3=C_3\rtimes C_2$ as its image.
Thus $\textup{PSL}(2,\mathbb{Z}_2)$ is a virtually pro-$2$ CA-group which is neither pro-$2$ nor virtually abelian, and it shows that both cases of \hyperref[t:virt_pro-p_intro]{Theorem~\ref{t:virt_pro-p_intro}} can occur. 

\medskip
Finally, we prove: 

\begin{customthm}{\ref{t:simple_intro}} 
Let $G$ be an infinite profinite CA-group. If $G$ is not prosoluble, then $G/O_2(G)$ is a finite almost simple group.
\end{customthm}

\begin{proof} 
In virtue of \hyperref[t:virt_pronilp_intro]{Theorem~\ref{t:virt_pronilp_intro}},
\hyperref[t:virt_abel_intro]{Theorem~\ref{t:virt_abel_intro}}, and \hyperref[p:prosoluble]{Proposition~\ref{p:prosoluble}},
it follows that $G$ is virtually pro-$2$. 
Let $M/O_2(G)$ be a minimal normal subgroup of $G/O_2(G)$.
For an odd prime $p\in\pi(M)$, let $P$ be a $p$-Sylow subgroup of $M$. 
By the Frattini argument, $G=MN_G(P)$. 
Since $P$ is finite, $N_G(P)$ has a subgroup of finite index which centralizes $P$. 
Hence $N_G(P)$ is virtually abelian. 
If $N_G(P)$ is finite, the classification of finite CA-groups shows that $N_G(P)$ is soluble. 
If $N_G(P)$ is infinite,  \hyperref[t:virt_abel_intro]{Theorem~\ref{t:virt_abel_intro}} shows that $N_G(P)$ is prosoluble. 
Thus, in all cases $N_G(P)$ is prosoluble.

Suppose that $M/O_2(G)$ is abelian. 
In that case the equality $G=MN_G(P)$ guarantees that $G$ is prosoluble. 
This is a contradiction and so $G/O_2(G)$ does not have normal soluble subgroups.

Therefore $M/O_2(G)$ is a direct product of isomorphic non-abelian simple groups. 
Suppose that $M/O_2(G)$ is not simple. 
Then, for any  odd prime $p\in\pi(M)$ the $p$-Sylow subgroup of $M$ is not cyclic. 
In view of \hyperref[v4]{Lemma~\ref{v4}} this is a contradiction. 
Hence $M/O_2(G)$ is simple. 
Taking into account that $G/O_2(G)$ does not have nontrivial normal soluble subgroups and putting this together with the fact that $G=MN_G(P)$, where $N_G(P)$ is prosoluble, 
we conclude that $M/O_2(G)$ is the unique minimal normal subgroup of $G/O_2(G)$. 
This means that $G/O_2(G)$ is almost simple.
\end{proof}

\section{Profinite CSA-groups}
\label{s:csa}

In the studies of fully residually free groups, limit groups,
and discriminating groups, 
the class of the \emph{conjugately separated abelian groups}
(\textit{i.e.}, CSA-groups) plays a fundamental role
(see \cite{MR:96} and \cite{FGRS:16}).

Observe that by \cite[Thm.~3.5]{FGRS:16} finite CSA-groups are abelian. Hence the really interesting CSA-groups are
the (infinite) non-abelian CSA-groups.

Let us prove our last result.

\begin{customthm}{\ref{t:CSA_intro}}  
Let $G$ be a profinite CSA-group. Then  one of the following holds:

\begin{enumerate} 

\item[(i)] $G$ is abelian;

\item[(ii)] G is pro-$p$;

\item[(iii)] for some prime $p$ the quotient group  $G/O_p(G)$ is  cyclic of order coprime to~$p-1$.

\end{enumerate}
\end{customthm}
\begin{proof} 
Assume that $G$ is non-abelian. 
By \cite[Thm.~3.5]{FGRS:16} finite CSA-groups are abelian, so $G$ is infinite. 
Thus \hyperref[t:virt_pronilp_intro]{Theorem~\ref{t:virt_pronilp_intro}} tells us that $G$ is either virtually abelian or virtually pro-$p$ for some prime $p$. 
It is straightforward that a CSA-group having a nontrivial normal abelian subgroup is abelian. 
Hence, $G$ is virtually pro-$p$.

Set $P=O_p(G)$. Let us show first that $G/P$ is cyclic. 
Since  a Frobenius group is not CSA, in view of \hyperref[t:virt_pro-p_intro]{Theorem~\ref{t:virt_pro-p_intro}} it is now sufficient to show that $G/P$ is abelian. 
Suppose this is false and let $G$ be a counter-example with $|G/P|$ as small as possible. 
Thus, the proper subgroups of $G/P$ are abelian and so $G$ is prosoluble.
\hyperref[t:virt_pro-p_intro]{Theorem~\ref{t:virt_pro-p_intro}} tells us that $G/P$ is a metacyclic Frobenius group. 
Choose an element $b\in G$ such that the image of $b$ in $G/P$ generates the Frobenius kernel of $G/P$. 
Because of minimality of $|G/P|$ the element $b$ has prime order $q\neq p$. 
By the Frattini argument $G=PN_G(\langle b\rangle)$. 
Now it is sufficient to show that $N_G(\langle b\rangle)$ is abelian. 
Obviously, $N_G(\langle b\rangle)$ has a nontrivial normal abelian subgroup. 
Thus, by the remark in the preceding paragraph $N_G(\langle b\rangle)$ is abelian. 
Therefore the equality $G=PN_G(\langle b\rangle)$ implies that $G/P$ is abelian, a contradiction. 
Thus $G/P$ is cyclic.

Now suppose that $G$ is not pro-$p$. We show that $(|G/P|,p-1)=1$.  
Choose a $p'$-Hall subgroup $B$ in $G$. 
We already know that $B$ is cyclic and $G=PB$. 
Suppose on the contrary there exists  a subgroup $C$ of prime order $q$ in $B$ such that $q$ divides $p-1$. 
In virtue of \hyperref[cc]{Lemma~\ref{cc}}(v)  $C$ acts nontrivially on $P/\Phi(P)$. 
It follows that the semidirect product $[P/\Phi(P),C]\rtimes C$ is a Frobenius group. 
Since $q$ divides $p-1$, we can choose a $C$-invariant cyclic subgroup $D$ in $[P/\Phi(P),C]$. 
Note that $DC$ is a Frobenius group with a cyclic kernel.  
Then by \cite[Thm.~5.1]{GH:11}  $P$ contains a $C$-invariant procyclic subgroup  $H$ such that $HC$ is a Frobenius profinite group. 
Since a CSA-group having a nontrivial normal abelian subgroup is abelian, we obtain a contradiction.
\end{proof}

\section{Open questions and concluding remarks}

We finish the paper with open questions that indicate directions of further studies. 

\hyperref[t:virt_pronilp_intro]{Theorem~\ref{t:virt_pronilp_intro}} shows that a profinite CA-group is virtually abelian or virtually pro-$p$.
Pro-$p$ groups that appear as an open subgroups of virtually pro-$p$ (but not pro-$p$) CA-groups must be more restrictive (for CSA even more restrictive).  

\begin{question} What  pro-$p$ groups can appear as open subgroups of virtually pro-$p$ \textup{(}but not pro-$p$\textup{)} CA-groups and CSA-groups?\end{question}

\hyperref[t:simple_intro]{Theorem~\ref{t:simple_intro}} shows that non-prosolubility of a profinite CA-group forces it to be virtually pro-$2$.
We do not know however an example of such a group.

\begin{question} Does there exist a profinite CA-group which is not prosoluble? If yes, which finite simple groups occur as sections of such groups?\end{question}

\hyperref[t:CSA_intro]{Theorem~\ref{t:CSA_intro}} shows that non-abelian profinite CSA-groups are pro-$p$-by-cyclic.
We however do not know an example of non-abelian non-pro-$p$ profinite CSA-group.

\begin{question} Does there exist a non-abelian profinite CSA-group which is not \mbox{pro-$p$}?\end{question}

We do not know a concrete example of a group $G$ satisfying the hypothesis of \hyperref[t:virt_abel_intro]{Theorem~\ref{t:virt_abel_intro}} with $G/N\cong\mathbf{Q}_{2^n}$ for $n\geq 4$.
Hence: 
 
\begin{question} Does there exist an example of a profinite CA-group $G$ having an abelian normal subgroup $N$ with $G/N\cong \mathbf{Q}_{2^n}$ for $n\geq 4$?\end{question}

\phantomsection
\section*{Acknowledgements}
\label{s:acknowl}
The authors thank Alexandre Zalesski for valuable advice on parts of the paper.
The authors are also grateful for financial support from CNPq and FAPDF.

\label{bibliography}


\begin{thebibliography}{XXXXX}

\small


\bibitem[Bol63]{Bolker:63}
{E. Bolker},
\emph{Inverse limits of solvable groups},
{Proc. Amer. Math. Soc.},
{\bf 14}
{(1963)}, 
{147--152}.

\bibitem[BSW58]{BSW:58}
{R. Brauer, M. Suzuki, and G. Wall}, 
\emph{A characterization of the one-dimensional unimodular projective groups over finite fields}, 
{Illinois J. Math.}, 
{\bf 2}
{(1958)}, 
{718--745}.

\bibitem[CR62]{CR:62}
{C. Curtis and I. Reiner},
\emph{Representation Theory of Finite Groups and Associative Algebras},
{Interscience Publishers},
{New York},
{1962}.




\bibitem[FJ05]{FJ}
{M. Fried and M. Jarden},
\emph{Field arithmetic},
{Springer-Verlag},
{Berlin},
{2005}.

\bibitem[FR08]{FR:08}
{B. Fine and G. Rosenberger},
{Reflections on commutative transitivity},
in
\emph{Aspects of infinite groups},
{112--130},
{World Sci. Publ.}, 
{2008}.


\bibitem[FGRS16]{FGRS:16}
{B. Fine, A. Gaglione, G. Rosenberger, and D. Spellman},
\emph{On CT and CSA groups and related ideas},
{J. Group Theory}, 
{\bf 19} 
{(2016)}, 
{923--940}.

\bibitem[FT63]{FT:63}
{W. Feit and J. Thompson}, 
\emph{Solvability of groups of odd order}, 
{Pacific J. Math.},
{\bf 13} 
{(1963)}, 
{775--1029}.


\bibitem[Gor80]{Gorenstein:80}
{D. Gorenstein},
\emph{Finite groups},
{Chelsea Publishing Co.},
{New York},
{1980}.

\bibitem[GH11]{GH:11} {Robert M. Guralnick, D. Haran},
{\emph Frobenius subgroups of free profinite products},
{Bull. of the London Math. Society} {\bf 43}, 
{(2011)},
 {1141--1150}.



\bibitem[Hig57]{Higman:57}
{G. Higman},
\emph{Groups and rings having automorphisms without non-trivial fixed elements},
{J. London Math. Soc.}, 
{\bf 32} 
{(1957)},
{321--334}. 

\bibitem[KhMSh14]{KhMSh:14}
{E. Khukhro, N. Makarenko, and P. Shumyatsky},
\emph{Frobenius groups of automorphisms and their fixed points}, 
{Forum Math.}, 
{\bf 26} 
{(2014)}, 
{73--112}.

\bibitem[KhSh15]{KhSh:15}
{E. Khukhro and P. Shumyatsky},
\emph{Nonsoluble and non-$p$-soluble length of finite groups},
{Israel J. Math.},
{\bf 207}
{(2015)},
{507--525}.





\bibitem[KZ10]{KZ:10}
{D. Kochloukova and P. Zalesskii},
\emph{Fully residually free pro-{$p$} groups},
{J. Algebra},
{\bf 324} 
{(2010)},
{782--792}.

\bibitem[KZ11]{KZ:11}
{D. Kochloukova and P. Zalesskii},
\emph{On pro-p analogues of limit groups via extensions of centralizers},
{Math. Z.},
{\bf 267}
{(2011)},
{109--128}.


\bibitem[LS77]{LS:77}
{R. Lyndon and P. Schupp},
\emph{Combinatorial group theory},
{Springer-Verlag},
{Berlin--New York},
{1977}.

\bibitem[MR96]{MR:96}
{A. Myasnikov and V. Remeslennikov},
\emph{Exponential groups. II.
Extensions of centralizers and tensor completion of CSA groups},
{Internat. J. Algebra and Comput.},
{\bf 6}
{(1996)},
{687--711}.


\bibitem[RZ00]{RZ:00}
{L. Ribes and P. Zalesskii},
\emph{Profinite groups},
{Springer},
{Berlin},
{2000}.


\bibitem[Rob96]{Robinson:96}
{D. Robinson},
\emph{A Course in the Theory of Groups},
{Springer},
{New York},
{1996}.

\bibitem[Ros09]{Rose:09}
{H. Rose},
\emph{A Course on Finite Groups},
{Springer},
{London},
{2009}.



\bibitem[Ser08]{Serre:08}
{J-P. Serre},
{Three letters to Walter Feit on group representations and quaternions}, 
{J. Algebra}, 
{\bf 319}
{(2008)},
{549--557}.

\bibitem[Shu98]{Shumyatsky:98} 
{P. Shumyatsky}, 
\emph{Centralizers in groups with finiteness conditions}, 
{J. Group Theory}, 
{\bf 1}
{(1998)},
275--282.

\bibitem[Suz57]{Suzuki:57}  
{M. Suzuki},
\emph{The nonexistence of a certain type of simple groups of odd order},
{Proc. Amer. Math. Soc.},
{\bf 8}
{(1957)}, 
{686--695}.

\bibitem[Suz86]{Suzuki:86}
{M. Suzuki},
\emph{Group Theory. II},
{Springer-Verlag},
{New York},
{1986}.

\bibitem[Tho59]{Thompson:59}
{J. Thompson},
\emph{Finite groups with fixed-point-free automorphisms of prime order},
{Proc. Nat. Acad. Sci. U.S.A.},
{\bf 45}
{(1959)},
{578--581}.

\bibitem[Wei25]{Weisner:25}  
{L. Weisner},
\emph{Groups in which the normaliser of every element except identity is abelian},
{Bull. Amer. Math. Soc.},
{\bf 31}
{(1925)},
{413--416}.

\bibitem[Wil17]{Wilkes:17}
{G. Wilkes},
\emph{Virtual pro-p properties of 3-manifold groups}, 
{J. Group Theory}, {\bf 20} (2017), 999-1023.

\bibitem[Wil83]{Wilson:83}
{J. Wilson},
\emph{On the structure of compact torsion groups},
{Monatsh. Math.},
{\bf 96}
{(1983)},
{57--66}.

\bibitem[Wu98]{Wu:98}  
{Yu-Fen Wu},
\emph{Groups in which commutativity is a transitive relation},
{J. Algebra},
{\bf 207}
{(1998)},
{165--181}.


\bibitem[ZZ17]{ZZ:17}
{P. Zalesskii and T. Zapata},
\emph{Profinite extensions of centralizers and the profinite completion of limit groups},
to appear in
{Rev. Mat. Iberoam.}
\url{http://arxiv.org/abs/1711.01500}.
%

\bibitem[Zel92]{Zelmanov:92}
{E. Zelmanov},
\emph{On periodic compact groups},
{Israel J. Math.},
{\bf 77}
{(1992)},
{83--95}.


\end{thebibliography}
\end{document}